\documentclass{amsart}%
\usepackage{amssymb}
\usepackage{cite}
\usepackage{amsfonts}
\usepackage{amsmath}
\usepackage{graphicx}%
\providecommand{\U}[1]{\protect\rule{.1in}{.1in}}
\newtheorem{theorem}{Theorem}
\theoremstyle{plain}

\newtheorem{lemma}{Lemma}

\newtheorem{remark}{Remark}

\numberwithin{equation}{section}
\begin{document}
\title[$q$-generalizations of Meyer-K\"{o}n\.{i}g and Zeller Operators]{Korovkin type Approximation of Abel transforms of $q$-Meyer-K\"{o}n\.{i}g and
Zeller operators}
\author{D\.{i}lek S\"oylemez}
\email{dsoylemez@ankara.edu.tr}
\author{Mehmet \"{U}nver}
\address{Ankara University, Elmada\u{g} Vocational School, Department of Computer Programming,\\
06780 Ankara, Turkey\\
Ankara University, Faculty of Science, Department of Mathematics,\\
06100 Ankara TURKEY}
\email{munver@ankara.edu.tr}

\begin{abstract}
In this paper we investigate some Korovkin type approximation properties of
the $q$-Meyer-K\"{o}nig and Zeller operators and Durrmeyer variant of \ the
$q$-Meyer-K\"{o}nig and Zeller operators via Abel summability method which is
a sequence-to-function transformation and which extends the ordinary
convergence. We show that the approximation results obtained in this paper are
more general than some previous results. Finally, we obtain the rate of Abel
convergence for the corresponding operators.

\end{abstract}
\maketitle

\section{Preliminaries}

Korovkin type approximation theory aims to provide some simple criteria for
the convergence of a sequence of positive linear operators in some senses
\cite{korovkin}. There is a number of main motivations in the theory. One of
them is obtaining some suitable conditions for the convergence of arbitrary
sequence of positive linear operators acting from one certain space to another
one. Next motivation is studying some particular conditions for convergence
for certain sequence of positive linear operators by using known criteria (see
e.g., \cite{altomare-cambiti}). It is also possible to introduce the
summability theory whose main idea is to make a non-convergent sequence or
series to converge in some general senses whenever the sequence of positive
linear operators does not converge in the ordinary sense. The leading study
with this motivation gives criteria for the statistical convergence of a
sequence of positive linear operators over $C[a,b]$, the space of all real
continuous functions defined on the interval $[a,b]$ \cite{orhan}. Following
that study many authors have given several approximation results via
summability theory \cite{atlihan, dogru,unver2}.

In 1987, Lupa\c{s} \cite{lupas} introduced the first $q$-analogue of Bernstein
operators and investigated its approximation and shape preserving properties.
In 1997, Phillips \cite{phillips1} defined another $q$-generalization of
Bernstein operators. Afterwards, many generalizations of positive linear
operators based on $q$-integers were introduced and studied by several
authors, for example, we refer the readers to \cite{acu-barbosu-sofenia},
\cite{agratini-dogru}, \cite{aral-gupta2013}, \cite{dogru-gupta},
\cite{mishra1}, \cite{mishra2}, \cite{mursaleen}, \cite{ostrovska2003},
\cite{soylemez-tunca-aral}, \cite{soylemez}, \cite{soylemez2},
\cite{srivastava-mur-alo}.

Now, let us recall some notations from $q$- analysis \cite{phillips-kitap}:
\newline For any fixed real number $q>0$, the $q$-integer $\left[  n\right]  $
is defined by%

\[
\left[  n\right]  :=\left[  n\right]  _{q}=\sum_{k=1}^{n}q^{k-1}=\left\{
\begin{array}
[c]{lll}%
\frac{1-q^{n}}{1-q} & , & q\neq1\\
&  & \\
n & , & q=1
\end{array}
\right.  ,
\]
where $n$ is a positive integer and $\left[  0\right]  =0,$ the $q$- factorial
$\left[  n\right]  !$ of $[n]$ is given with%

\[
\left[  n\right]  !:=\left\{
\begin{array}
[c]{lll}%
\prod\limits_{k=1}^{n}[k] & , & n=1,2,...\\
&  & \\
1 & , & n=0.
\end{array}
\right.
\]
\newline For integers $n\geq r\geq0,$ the$\ q$-binomial coefficient is defined by%

\[%
%TCIMACRO{\QATOPD{[}{]}{n}{r}}%
%BeginExpansion
\genfrac{[}{]}{0pt}{}{n}{r}%
%EndExpansion
_{q}=\frac{\left[  n\right]  _{q}!}{\left[  r\right]  _{q}!\left[  n-r\right]
_{q}!}%
\]
and $q$-shifted factorial is defined by%

\[
\left(  t;q\right)  _{n}:=\left\{
\begin{array}
[c]{lll}%
1 & , & n=0\\
&  & \\
\prod\limits_{j=0}^{n-1}\left(  1-tq^{j}\right)  & , & n=1,2,...
\end{array}
\right.  .
\]
\bigskip Thomae \cite{thomae} introduced the $q$-integral of function $f$
defined on the interval $\left[  0,a\right]  $ as follows:%
\[%
%TCIMACRO{\dint \limits_{0}^{a}}%
%BeginExpansion
{\displaystyle\int\limits_{0}^{a}}
%EndExpansion
f\left(  t\right)  d_{q}t:=a\left(  1-q\right)  \sum\limits_{n=0}^{\infty
}f\left(  aq^{n}\right)  q^{n},\text{ \ \ \ \ \ }0<q<1\text{.}%
\]
Finally, the $q$-beta function \cite{thomae} is defined by%
\[
B_{q}\left(  m,n\right)  =%
%TCIMACRO{\dint \limits_{0}^{1}}%
%BeginExpansion
{\displaystyle\int\limits_{0}^{1}}
%EndExpansion
t^{m-1}\left(  qt;q\right)  _{n-1}d_{q}t\text{.}%
\]

The original Meyer-K\"{o}nig and Zeller operators\ were introduced for $f\in
C\left[  0,1\right]  $ in 1960 $\left(  \text{see \cite{meyer-k-z}}\right)  .$
Later, Cheney and Sharma \cite{cheney-sharma} rearranged these operators as
follows:%
\begin{equation}
M_{n}\left(  f;x\right)  =\left\{
\begin{array}
[c]{ccc}%
\left(  1-x\right)  ^{n+1}\sum\limits_{k=0}^{\infty}f\left(  \frac{k}%
{n+k}\right)  \left(
\begin{array}
[c]{c}%
n+k\\
k
\end{array}
\right)  x^{k} & , & x\in\left[  0,1\right) \\
&  & \\
f\left(  1\right)  & , & x=1.
\end{array}
\right.  \label{1b}%
\end{equation}
Trif \cite{trif} defined the $q$- generalization of the Meyer-K\"{o}nig and
Zeller operators as
\begin{equation}
M_{n}^{q}\left(  f;x\right)  =\left\{
\begin{array}
[c]{ccc}%
\prod\limits_{j=0}^{n}\left(  1-q^{j}x\right)  \sum\limits_{k=0}^{\infty
}f\left(  \frac{\left[  k\right]  }{\left[  n+k\right]  }\right)  \left[
\begin{array}
[c]{c}%
n+k\\
k
\end{array}
\right]  x^{k} & , & x\in\left[  0,1\right) \\
&  & \\
f\left(  1\right)  , & , & x=1.
\end{array}
\right.  \label{q-mkz}%
\end{equation}
In \cite{trif}, the author studied Korovkin type approximation properties,
calculated the rate of convergence and also gave a result for monotonicity
properties of these operators. Heping \cite{heping} proved some approximation
results for the operators $\left(  M_{n}^{q}f\right)  $ using $q$%
-hypergeometric series. Another $q$-generalization of the classical
Meyer-K\"{o}nig and Zeller operators can be found in \cite{dogru-duman}.
Besides, Durrmeyer variant of the $q$-Meyer K\"{o}nig and Zeller operators
\cite{govil-gupta} is introduced for$\ f\in C\left[  0,1\right]  $,
$x\in\left[  0,1\right]  $, $\ n\in\mathbb{N}$ and $\alpha>0,q\in\left(
0,1\right]  \ $ as follows:%
\begin{equation}
D_{n}^{q}\left(  f;x\right)  =\left\{
\begin{array}
[c]{lll}%
\begin{array}
[c]{l}%
%TCIMACRO{\tsum \limits_{k=1}^{\infty}}%
%BeginExpansion
{\textstyle\sum\limits_{k=1}^{\infty}}
%EndExpansion
m_{n,k,q}\left(  x\right)
%TCIMACRO{\tint \limits_{0}^{1}}%
%BeginExpansion
{\textstyle\int\limits_{0}^{1}}
%EndExpansion
\tfrac{1}{B_{q}\left(  n,k\right)  }t^{k-1}\left(  qt;q\right)  _{n-1}f\left(
t\right)  d_{q}t\\
+m_{n,0,q}\left(  x\right)  f\left(  0\right)
\end{array}
& , & x\in\left[  0,1\right) \\
& , & x=1,\\
f\left(  1\right)  &  &
\end{array}
\right.  , \label{durr-q-mkz}%
\end{equation}
where%
\[
m_{n,k,q}\left(  x\right)  =\left(  x;q\right)  _{n+1}\left[
\begin{array}
[c]{c}%
n+k\\
k
\end{array}
\right]  x^{k}.
\]
The authors investigated some approximation properties with the help of
well-known Korovkin's theorem and compute the rate of convergence for these
operators in terms of the second-order modulus of continuity
\cite{govil-gupta}.

Throughout this paper, we study with the sequence $(q_{n})$ such that
$0<q_{n}\leq1$ and $q_{0}=0$ and we define $M_{0}^{q}f=D_{0}^{q}f=0$ for any
$f\in C[0,1]$. It is well known that if the classical conditions
\begin{equation}
\lim\limits_{n\rightarrow\infty}q_{n}^{n}=1\text{ and }\lim
\limits_{n\rightarrow\infty}\frac{1}{\left[  n\right]  }=0 \label{1c}%
\end{equation}
hold, then for each $f\in C[0,1]$ the sequences $(M_{n}^{q}f)$ and $(D_{n}%
^{q}f)$ converge uniformly to $f$ over $[0,1]$ (see \cite{govil-gupta, trif}).
Furthermore, we use the norm of the space $C[a,b]$ defined for any $f\in
C[a,b]$ by
\begin{equation}
||f||:=\sup_{a\leq x\leq b}|f(x)|.
\end{equation}

In the present paper, taking into account the Abel convergence we obtain some
approximation results for the $q$-Meyer-K\"{o}nig and Zeller operators and
Durrmeyer variant of $q$-Meyer-K\"{o}nig and Zeller operators. We also study
the rate of the convergence of these operators. We also show that the results
obtained in this paper are stronger than some previous ones.

Let $x=(x_{j})$ be a real sequence. If the series%
\begin{equation}
\sum_{j=0}^{\infty}x_{j}y^{j} \label{1d1}%
\end{equation}
is convergent for any $y\in(0,1)$ and%
\[
\lim_{y\rightarrow1^{-}}(1-y)\sum_{j=0}^{\infty}x_{j}y^{j}=\alpha
\]
then $x$ is said to be Abel convergent to real number $\alpha$ \cite{powell}.
Korovkin type approximation via Abel convergence and other power series
methods may be found in
\cite{unver-atlihan,braha1,braha2,ilknur,emre,tugba-emre,unver,unver2,dilek-mehmet}%
.

The fact given in the following remark helps us throught the paper:

\begin{remark}
\label{remark1}Let $\left(  f_{n}\right)  $ be a sequence in $C[0,1]$. If
there exists a positive integer $n_{0}$ such that%
\[
\lim_{y\rightarrow1^{-}}(1-y)\left\Vert \sum_{j=n_{0}}^{\infty}f_{j}%
y^{j}\right\Vert =0
\]
then it is not difficult to see that%
\[
\lim_{y\rightarrow1^{-}}(1-y)\left\Vert \sum_{j=0}^{\infty}f_{j}%
y^{j}\right\Vert =0,
\]
i.e., while studying the Abel convergence finitely many terms do not make
sense as with the ordinary convergence.
\end{remark}

Before studying the announced approximation properties of the operators, we
recall some well-known lemmas:

\begin{lemma}
\label{lemma qmkz} \cite{trif} Let $n\geq3$ be a positive integer. Then the
following hold for the operators $\left(  \ref{q-mkz}\right)  :$%
\begin{equation}
M_{n}^{q}\left(  e_{0};x\right)  =1 \label{1h}%
\end{equation}%
\begin{equation}
M_{n}^{q}\left(  e_{1};x\right)  =x \label{1k}%
\end{equation}%
\begin{equation}
x^{2}\leq M_{n}^{q}\left(  e_{2};x\right)  \leq\frac{x}{\left[  n-1\right]
}+x^{2} \label{1l}%
\end{equation}
where $e_{i}(x)=x^{i}$ for $i=0,1,2$.
\end{lemma}

\begin{lemma}
\label{lemma durrmeyer} \cite{govil-gupta} Let $n\geq3$ be a positive integer.
Then the following hold for the operators $\left(  \ref{durr-q-mkz}\right)  :$%
\begin{equation}
D_{n}^{q}\left(  e_{0};x\right)  =1 \label{1m}%
\end{equation}%
\begin{equation}
D_{n}^{q}\left(  e_{1};x\right)  =x \label{1n}%
\end{equation}%
\begin{equation}
D_{n}^{q}\left(  e_{2};x\right)  =x^{2}+\frac{\left[  2\right]  x\left(
1-x\right)  \left(  1-q^{n}x\right)  }{\left[  n-1\right]  }-E_{n,q}\left(
x\right)  , \label{1o}%
\end{equation}
where%
\[
0\leq E_{n,q}\left(  x\right)  \leq\frac{x\left[  2\right]  \left[  3\right]
q^{n-1}}{\left[  n-1\right]  \left[  n-2\right]  }\left(  1-x\right)  \left(
1-qx\right)  \left(  1-q^{n}x\right)  .
\]

\end{lemma}

The following lemma can be proved easily:

\begin{lemma}
\label{lemma3}Let $n\geq3$ be a positive integer. Then we have%
\begin{align*}
D_{n}^{q}\left(  e_{2};x\right)  -x^{2}  &  \geq\frac{\left[  2\right]
x\left(  1-x\right)  \left(  1-q^{n}x\right)  }{\left[  n-1\right]  }\\
&  -\frac{x\left[  2\right]  \left[  3\right]  q^{n-1}}{\left[  n-1\right]
\left[  n-2\right]  }\left(  1-x\right)  \left(  1-qx\right)  \left(
1-q^{n}x\right) \\
&  \geq0.
\end{align*}

\end{lemma}

\section{Abel Transform of the sequence $(M_{n}^{q}$)}

In this section, we study Korovkin type approximation of the operators
$\left(  M_{n}^{q}\right)  $ defined with $\left(  \ref{q-mkz}\right)  $ by
considering the Abel method. Unver proved the following Korovkin-type theorem
via Abel method (\text{see \cite{unver}, Theorem 1).}

\begin{theorem}
\label{mehmet} Let $(L_{n})$ be a sequence of positive linear operators from
$C\left[  a,b\right]  \rightarrow B\left[  a,b\right]  $ such that $%
%TCIMACRO{\dsum \limits_{n=0}^{\infty}}%
%BeginExpansion
{\displaystyle\sum\limits_{n=0}^{\infty}}
%EndExpansion
\left\Vert L_{n}(e_{0})\right\Vert y^{n}<\infty$ for any $y\in(0,1)$. Then for
any $f\in C\left[  a,b\right]  $ we have%
\[
\lim_{y\rightarrow1^{-}}\left(  1-y\right)  \left\Vert \sum\limits_{n=0}%
^{\infty}\left(  L_{n}f-f\right)  y^{n}\right\Vert =0
\]
if and only if
\begin{equation}
\lim_{y\rightarrow1^{-}}\left(  1-y\right)  \left\Vert \sum\limits_{n=0}%
^{\infty}\left(  L_{n}e_{i}-e_{i}\right)  y^{n}\right\Vert =0,\text{ }i=0,1,2.
\label{2d}%
\end{equation}

\end{theorem}

We are now ready to prove the following theorem:

\begin{theorem}
\label{theorem2}If the sequence $\left(  \dfrac{1}{\left[  n-1\right]
}\right)  _{n=3}^{\infty}$ is Abel null then for each $f\in C\left[
0,1\right]  $ we have%
\[
\lim_{y\rightarrow1^{-}}\left(  1-y\right)  \left\Vert \sum\limits_{n=0}%
^{\infty}\left(  M_{n}^{q}f-f\right)  y^{n}\right\Vert =0.
\]

\end{theorem}

\begin{proof}
From Lemma \ref{lemma qmkz} we see that $%
%TCIMACRO{\dsum \limits_{n=0}^{\infty}}%
%BeginExpansion
{\displaystyle\sum\limits_{n=0}^{\infty}}
%EndExpansion
\left\Vert M_{n}^{q}(e_{0})\right\Vert y^{n}<\infty$ for any \linebreak%
$y\in(0,1)$. If we consider Theorem \ref{mehmet}, it suffices to show that
$\left(  \ref{2d}\right)  $ holds for $(M_{n}^{q}).\ $ Now, considering Lemma
\ref{lemma qmkz} , one can get for $i=0,1$ that
\[
\lim_{y\rightarrow1^{-}}\left(  1-y\right)  \left\Vert \sum\limits_{n=0}%
^{\infty}\left(  M_{n}^{q}e_{i}-e_{i}\right)  y^{n}\right\Vert =0.
\]
Moreover, using $\left(  \ref{1l}\right)  $, we have for $n\geq3$ that%
\[
0\leq M_{n}^{q}\left(  e_{2};x\right)  -x^{2}\leq\frac{x}{\left[  n-1\right]
}%
\]
which gives%
\begin{align*}
0  &  \leq\left(  1-y\right)  \left\Vert \sum\limits_{n=3}^{\infty}\left(
M_{n}^{q}e_{2}-e_{2}\right)  y^{n}\right\Vert \\
&  \leq\left(  1-y\right)  \sup_{0\leq x\leq1}\sum\limits_{n=3}^{\infty}%
\frac{x}{\left[  n-1\right]  }y^{n}\\
&  \leq\left(  1-y\right)  \sum\limits_{n=3}^{\infty}\frac{y^{n}}{\left[
n-1\right]  }.
\end{align*}
Finally we have%
\[
\lim_{y\rightarrow1^{-}}\left(  1-y\right)  \left\Vert \sum\limits_{n=3}%
^{\infty}\left(  M_{n}^{q}e_{2}-e_{2}\right)  y^{n}\right\Vert =0.
\]
Hence, considering Remark \ref{remark1} we get%
\[
\lim_{y\rightarrow1^{-}}\left(  1-y\right)  \left\Vert \sum\limits_{n=0}%
^{\infty}\left(  M_{n}^{q}e_{2}-e_{2}\right)  y^{n}\right\Vert =0
\]
which concludes the proof.
\end{proof}

The following remark proves that the conditions of Theorem \ref{theorem2} are
weaker than the classical conditions:

\begin{remark}
It is not difficult to see that the classical conditions $\left(
\ref{1c}\right)  $ entail that the sequence $\left(  \dfrac{1}{\left[
n-1\right]  }\right)  _{n=3}^{\infty}$ is Abel null. Conversely, if we define
the sequence $(q_{n})$ with%
\[
q_{n}:=\left\{
\begin{array}
[c]{lll}%
0 & , & n\text{ is a perfect cube}\\
1 & , & \text{otherwise}%
\end{array}
\right.
\]
then $(q_{n})$ does not satisfy the classical conditions. Besides, we have for
any $n\geq3$ that%
\[
\frac{1}{\left[  n-1\right]  }=\left\{
\begin{array}
[c]{lll}%
1 & , & n\text{ is a perfect cube}\\
\dfrac{1}{n-1} & , & \text{otherwise.}%
\end{array}
\right.
\]
Now, since the sequence $\left(  \dfrac{1}{\left[  n-1\right]  }\right)
_{n=3}^{\infty}$ is bounded and statistically convergent to zero it is Abel
null \cite{powell,sch.}.
\end{remark}

\section{Abel Transform of the operators $\left(  D_{n}^{q}\right)  $}

In this section, we study Korovkin type approximation of the operators
$\left(  D_{n}^{q}\right)  $ defined with $\left(  \ref{durr-q-mkz}\right)  $
by considering the Abel method as well.

\begin{theorem}
\label{theorem1}If the sequence $\left(  \dfrac{\left[  2\right]  }{\left[
n-1\right]  }\right)  _{n=3}^{\infty}$ is Abel null then for each $f\in
C\left[  0,1\right]  $ we have%
\[
\lim_{y\rightarrow1^{-}}\left(  1-y\right)  \left\Vert \sum\limits_{n=0}%
^{\infty}\left(  D_{n}^{q}f-f\right)  y^{n}\right\Vert =0.
\]

\end{theorem}

\begin{proof}
Lemma \ref{lemma durrmeyer} implies that $%
%TCIMACRO{\dsum \limits_{n=0}^{\infty}}%
%BeginExpansion
{\displaystyle\sum\limits_{n=0}^{\infty}}
%EndExpansion
\left\Vert D_{n}^{q}(e_{0})\right\Vert y^{n}<\infty$. From\ Theorem
\ref{mehmet}, it sufficies to show that $\left(  \ref{2d}\right)  $ holds for
$(D_{n}^{q}).$ Using $\left(  \ref{1m}\right)  $ and $\left(  \ref{1n}\right)
,$ we obtain for $i=0,1$ that%
\[
\lim_{y\rightarrow1^{-}}\left(  1-y\right)  \left\Vert \sum\limits_{n=0}%
^{\infty}\left(  D_{n}^{q}e_{i}-e_{i}\right)  y^{n}\right\Vert =0.
\]
On the other hand from $\left(  \ref{1o}\right)  $ and Lemma \ref{lemma3} we
have for any $n\geq3$ that%
\begin{align*}
&  \frac{\left[  2\right]  x\left(  1-x\right)  \left(  1-q^{n}x\right)
}{\left[  n-1\right]  }-\frac{x\left[  2\right]  \left[  3\right]  q^{n-1}%
}{\left[  n-1\right]  \left[  n-2\right]  }\left(  1-x\right)  \left(
1-qx\right)  \left(  1-q^{n}x\right) \\
&  \leq D_{n}^{q}\left(  e_{2};x\right)  -x^{2}\\
&  \leq\frac{\left[  2\right]  x\left(  1-x\right)  \left(  1-q^{n}x\right)
}{\left[  n-1\right]  }%
\end{align*}
which implies%
\begin{align*}
0  &  \leq\left(  1-y\right)  \left\Vert \sum\limits_{n=3}^{\infty}\left(
D_{n}^{q}e_{2}-e_{2}\right)  y^{n}\right\Vert \\
&  \leq\left(  1-y\right)  \sup_{0\leq x\leq1}\sum\limits_{n=3}^{\infty
}\left(  \frac{\left[  2\right]  x\left(  1-x\right)  \left(  1-q^{n}x\right)
}{\left[  n-1\right]  }\right)  y^{n}\\
&  \leq\left(  1-y\right)  \sum\limits_{n=3}^{\infty}\frac{\left[  2\right]
}{\left[  n-1\right]  }y^{n}.
\end{align*}
Now from the hypothesis we obtain%
\[
\lim_{y\rightarrow1^{-}}\left(  1-y\right)  \left\Vert \sum\limits_{n=3}%
^{\infty}\left(  D_{n}^{q}e_{2}-e_{2}\right)  y^{n}\right\Vert =0
\]
. Therefore, from Remark \ref{remark1} we can write%
\[
\lim_{y\rightarrow1^{-}}\left(  1-y\right)  \left\Vert \sum\limits_{n=0}%
^{\infty}\left(  D_{n}^{q}e_{2}-e_{2}\right)  y^{n}\right\Vert =0.
\]
which ends the proof.
\end{proof}

Following remark shows that the condition of Theorem \ref{theorem1} is weaker
than the classical conditions $\left(  \ref{1c}\right)  $:

\begin{remark}
Note that if the classical conditions $\left(  \ref{1c}\right)  $ hold then
condition of Theorem \ref{theorem1} holds. In fact, if $\lim
\limits_{n\rightarrow\infty}q_{n}^{n}=1$ and $\lim\limits_{n\rightarrow\infty
}\frac{1}{\left[  n\right]  _{n}}=0$ then we have%
\[
\lim_{n\rightarrow\infty}\frac{\left[  2\right]  }{\left[  n-1\right]  }%
=\lim_{n\rightarrow\infty}\frac{1+q_{n}}{\left[  n\right]  -q_{n}^{n-1}}=0.
\]
Therefore, it is Abel null. Conversely, consider the sequence $(q_{n})$ given
by%
\[
q_{n}:=\left\{
\begin{array}
[c]{lll}%
0 & , & n\text{ is a prime}\\
&  & \\
1 & , & \text{otherwise.}%
\end{array}
\right.
\]
We see that $(q_{n})$ does not satisfy the conditions of classical Korovkin
theorem. On the other hand, one can have for any $n\geq2$ that%
\[
\frac{1+q_{n}}{\left[  n\right]  -q_{n}^{n-1}}=\left\{
\begin{array}
[c]{lll}%
1 & , & n\text{ is a prime}\\
&  & \\
\dfrac{2}{n-1} & , & \text{otherwise.}%
\end{array}
\right.
\]
Thus the sequence $\left(  \dfrac{1+q_{n}}{\left[  n\right]  -q_{n}^{n-1}%
}\right)  _{n=2}^{\infty}$ is Abel convergent to zero (since it is bounded and
statistically null).
\end{remark}

\section{Rate of Abel convergence}

In this section, we compute the rate of the Abel convergence by means of the
modulus of continuity. The modulus of continuity of $\omega\left(
f,\delta\right)  $ is defined by%
\[
\omega\left(  f,\delta\right)  =\sup_{\substack{\left\vert x-y\right\vert
\leq\delta\\x,y\in\left[  0,1\right]  }}\left\vert f\left(  x\right)
-f\left(  y\right)  \right\vert
\]
It is well known that, for any $f$ $\in C\left[  a,b\right]  ,$%
\begin{equation}
\lim_{\delta\rightarrow0^{+}}\omega\left(  f,\delta\right)  =0 \label{3a}%
\end{equation}
and for any $\delta>0$%
\begin{equation}
\left\vert f\left(  x\right)  -f\left(  y\right)  \right\vert \leq
\omega\left(  f,\delta\right)  \left(  \frac{\left\vert x-y\right\vert
}{\delta}+1\right)  \label{3b}%
\end{equation}
and for all $c>0$%
\[
\omega(f,c\delta)\leq(1+\left\lfloor c\right\rfloor )\omega(f,\delta)
\]
where $\left\lfloor c\right\rfloor $ is the greatest integer less than or
equal to $c.$ Now we are ready to give the following lemma:

\begin{lemma}
\label{lemma4}For any $f\in C\left[  0,1\right]  $ we have%
\[
\left(  1-y\right)  \left\Vert \sum\limits_{n=3}^{\infty}\left(  D_{n}%
^{q}f-f\right)  y^{n}\right\Vert \leq2\omega\left(  f,\varphi\left(  y\right)
\right)  ,
\]
where%
\begin{equation}
\varphi\left(  y\right)  :=\left\{  \left(  1-y\right)  \sup_{0\leq x\leq
1}\sum\limits_{n=3}^{\infty}D_{n}^{q}\left(  t-x\right)  ^{2}y^{n}\right\}
^{\frac{1}{2}} \label{11}%
\end{equation}
and the series in (\ref{11}) is convergent for each $y\in(0,1)$.
\end{lemma}

\begin{proof}
By using (\ref{1m}), for any $f\in C\left[  0,1\right]  $ and any $\delta>0$
we can write%
\begin{align*}
\left\vert \sum\limits_{n=3}^{\infty}\left(  D_{n}^{q}\left(  f;x\right)
-f\left(  x\right)  \right)  y^{n}\right\vert  &  \leq\sum\limits_{n=3}%
^{\infty}\left(  D_{n}^{q}\left\vert f\left(  t\right)  -f\left(  x\right)
\right\vert ;x\right)  y^{n}\\
&  \leq\sum\limits_{n=3}^{\infty}D_{n}^{q}\left(  \omega\left(  f,\frac
{\left\vert t-x\right\vert }{\delta}.\delta\right)  ;x\right)  y^{n}\\
&  \leq\sum\limits_{n=3}^{\infty}D_{n}^{q}\left(  \left(  1+\left\lfloor
\dfrac{\left\vert t-x\right\vert }{\delta}\right\rfloor \right)  \omega\left(
f,\delta\right)  ;x\right)  y^{n}\\
&  \leq\omega\left(  f,\delta\right)  \sum\limits_{n=3}^{\infty}D_{n}%
^{q}\left(  1+\frac{\left(  t-x\right)  ^{2}}{\delta^{2}};x\right)  y^{n}\\
&  \leq\omega\left(  f,\delta\right)  \sum\limits_{n=3}^{\infty}D_{n}%
^{q}\left(  e_{0}\left(  t\right)  ;x\right)  y^{n}\\
&  +\frac{1}{\delta^{2}}\omega\left(  f,\delta\right)  \sum\limits_{n=3}%
^{\infty}D_{n}^{q}\left(  \left(  t-x\right)  ^{2};x\right)  y^{n}\\
&  \leq\omega\left(  f,\delta\right)  \left(  \frac{1}{1-y}-y-y^{2}\right)
+\\
&  +\frac{1}{\delta^{2}}\omega\left(  f,\delta\right)  \sum\limits_{n=3}%
^{\infty}D_{n}^{q}\left(  \left(  t-x\right)  ^{2};x\right)  y^{n}.\\
\end{align*}
Thus, we reach to\bigskip%
\begin{align*}
\left(  1-y\right)  \left\vert \sum\limits_{n=3}^{\infty}\left(  D_{n}%
^{q}\left(  f;x\right)  -f\left(  x\right)  \right)  y^{n}\right\vert  &
\leq\omega\left(  f;\delta\right) \\
&  +\frac{1}{\delta^{2}}\omega\left(  f;\delta\right)  \sum\limits_{n=3}%
^{\infty}D_{n}^{q}\left(  \left(  t-x\right)  ^{2};x\right)  y^{n}%
\end{align*}
Now if we take $\delta=\left\{  \left(  1-y\right)  \sup_{0\leq x\leq1}%
\sum\limits_{n=3}^{\infty}D_{n}^{q}\left(  \left(  t-x\right)  ^{2};x\right)
y^{n}\right\}  ^{\frac{1}{2}},$ we get%
\[
0\leq\left(  1-y\right)  \left\Vert \sum\limits_{n=3}^{\infty}\left(
D_{n}^{q}f-f\right)  y^{n}\right\Vert \leq2\omega\left(  f,\varphi\left(
y\right)  \right)
\]
which completes the proof.
\end{proof}

\begin{remark}
If the sequence $\left(  \dfrac{\left[  2\right]  }{\left[  n-1\right]
}\right)  _{n=2}^{\infty}$ is Abel summable (need not to be zero) then the
series in (\ref{11}) is convergent for each $y\in(0,1)$.
\end{remark}

Using Lemma \ref{lemma4} and Remark \ref{remark1} the following theorem which
gives the rate of the Abel convergence for $(D_{n}^{q})$ can be proved:

\begin{theorem}
Let $\varphi$ be defined as Lemma \ref{lemma4} If $\omega\left(
f,\varphi\left(  y\right)  \right)  =o(\mu(y))$ as $y\rightarrow1^{-}$ then we
have%
\[
\left(  1-y\right)  \left\Vert \sum\limits_{n=1}^{\infty}\left(  D_{n}%
^{q}f-f\right)  y^{n}\right\Vert =o(\mu(y))\text{ as }y\rightarrow1.
\]

\end{theorem}

The rate of Abel convergence for $\left(  M_{n}^{q}\right)  $ defined with
$\left(  \ref{q-mkz}\right)  $ can be proved by using the similar idea for
$(D_{n}^{q})$:

\begin{theorem}
If $\omega\left(  f,\phi\left(  y\right)  \right)  =o(\mu(y))$ as
$y\rightarrow1^{-}.$ then we have%
\[
\left(  1-y\right)  \left\Vert \sum\limits_{n=1}^{\infty}\left(  M_{n}%
^{q}f-f\right)  y^{n}\right\Vert =o(\mu(y))\text{ as }y\rightarrow1^{-}%
\]
where%
\[
\phi\left(  y\right)  :=\left\{  \left(  1-y\right)  \sup_{0\leq x\leq1}%
\sum\limits_{n=3}^{\infty}M_{n}^{q}\left(  t-x\right)  ^{2}y^{n}\right\}
^{\frac{1}{2}}.
\]

\end{theorem}

\end{document}